\documentclass[12pt]{amsart}

\usepackage{mathrsfs, amssymb,amsthm, amsfonts, amsmath}
\usepackage[all]{xy}
\usepackage{upgreek}
\usepackage{amsmath}

\newcounter{cequation}[section]

\newtheorem{theorem}[cequation]{Theorem}
\newtheorem*{theorem*}{Theorem}
\newtheorem{lemma}[cequation]{Lemma}
\newtheorem{corollary}[cequation]{Corollary}

\newtheorem{proposition}[cequation]{Proposition}

\theoremstyle{definition}
\newtheorem{example}[cequation]{Example}
\newtheorem{definition}[cequation]{Definition}
\newtheorem*{definition*}{Definition}

\theoremstyle{remark}
\newtheorem{remark}[cequation]{Remark}

\makeatletter\@addtoreset{equation}{section}
\makeatother

\makeatother

\newcommand{\QQ}{\mathbb{Q}}
\newcommand{\ZZ}{\mathbb{Z}}
\newcommand{\PP}{\mathbb{P}}
\newcommand{\KK}{\mathbb{K}}
\newcommand{\LL}{\mathbb{L}}
\newcommand{\MM}{\mathbb{M}}

\newcommand{\cS}{\mathcal{S}}

\newcommand{\mumu}{{\boldsymbol{\mu}}}

\newcommand{\Aut}{\operatorname{Aut}}
\newcommand{\Bir}{\operatorname{Bir}}
\newcommand{\PGL}{\operatorname{PGL}}
\newcommand{\GL}{\operatorname{GL}}
\newcommand{\SL}{\operatorname{SL}}
\newcommand{\Gal}{\operatorname{Gal}}

\newcommand{\Fix}{\operatorname{Fix}}

\def \ge {\geqslant}
\def \le {\leqslant}

\date{}

\title{Finite groups acting on Severi--Brauer surfaces}

\author{Constantin Shramov}

\address{Steklov Mathematical Institute of Russian Academy of Sciences, 8 Gubkina st.,
Moscow, 119991, Russia
\newline
National Research University Higher School of Economics, Laboratory of Algebraic Geometry, 6 Usacheva str., Moscow, 119048, Russia
}

\email{costya.shramov@gmail.com}

\begin{document}

\begin{abstract}
We classify finite groups that can act by automorphisms and
birational automorphisms on non-trivial Severi--Brauer surfaces
over fields of characteristic zero.
\end{abstract}

\maketitle
\tableofcontents

\section{Introduction}

A \emph{Severi--Brauer surface} over a field $\KK$ is a surface that becomes isomorphic to $\PP^2$ over the algebraic closure of~$\KK$.
Concerning birational automorphisms of non-trivial Severi--Brauer surfaces
(i.e. those that are not isomorphic to~$\PP^2$ over the base field)
the following is known.

\begin{theorem}[{\cite[Theorem~1.2(ii)]{Shramov-SB}}]
\label{theorem:J-3}
Let $S$ be a non-trivial Severi--Brauer surface over a field $\KK$ of characteristic zero.
Then every finite group acting by birational automorphisms of $S$ is either abelian, or contains a normal abelian subgroup
of index~$3$.
\end{theorem}

The following result was proved in \cite{Shramov-SB-short}.
We denote by $\mumu_n$ the cyclic group of order~$n$.

\begin{theorem}
\label{theorem:attained-p}
Let $p\equiv 1 \pmod 3$ be a prime number. Consider the non-trivial
homomorphism~\mbox{$\mumu_3\to\Aut(\mumu_p)$}, and let
$G_p\cong\mumu_p\rtimes\mumu_3$ be the corresponding semidirect product. Then
there exists a number field $\KK$ and a non-trivial Severi--Brauer surface $S$ over~$\KK$ such that
the group $\Aut(S)$ contains the group~$G_p$.
\end{theorem}

Due to the work of I.\,Dolgachev and V.\,Iskovskikh \cite{DI},
a classification of finite groups acting by birational
selfmaps of $\PP^2$ over an algebraically closed field of characteristic
zero is available.
The purpose of this paper is to give a
classification of finite groups acting by birational selfmaps on
non-trivial Severi--Brauer surfaces over fields of characteristic zero.
We will show that possible
finite subgroups of automorphism groups of non-trivial Severi--Brauer
surfaces are cyclic groups of certain orders, and semidirect product
of cyclic groups with~$\mumu_3$ of certain kind;
possible finite subgroups of birational
automorphism groups of non-trivial Severi--Brauer
surfaces are finite subgroups of automorphism groups
together with the group~$\mumu_3^3$.

To formulate a precise assertion, we need
to introduce additional notation.
Given a positive integer $n$, denote by $\mumu_n^*$ the multiplicative
group of residues modulo $n$ that are coprime to $n$.
Let $G$ be a semidirect product of $\mumu_n$ and $\mumu_3$
corresponding to a homomorphism
$$
\chi\colon \mumu_3\to\mumu_n^*\cong\Aut(\mumu_n).
$$
If~\mbox{$n=\prod p_i^{r_i}$}, where $p_i$ are pairwise
different primes congruent to $1$
modulo $3$,
we say that~$G$ is \emph{balanced} if the compositions of
$\chi$ with all natural projections $\mumu_n^*\to\mumu_{p_i^{r_i}}^*$
are injective.
An example of such a semidirect product is
the group $\mumu_3\cong\mumu_1\rtimes\mumu_3$ itself.
We point out that a balanced semidirect product of a given
order may be not unique.
The reader is referred to~\S\ref{section:semi-direct} for details.

\begin{theorem}\label{theorem:main}
Let $n$ be a positive integer, and let $G$ be a finite group.
The following assertions hold.
\begin{itemize}
\item[(i)] Let $n$ be a positive integer.
There exists a field $\KK$ of characteristic zero and
a non-trivial Severi--Brauer surface over $\KK$ such that the group
$\Bir(S)$ of birational automorphisms of $S$ contains an element
of order $n$
if and only if~\mbox{$n=3^r\prod p_i^{r_i}$},
where $p_i$ are primes congruent to $1$
modulo $3$, and $r\le 1$
(in other words, $n$
is not divisible by $9$ and not divisible by primes congruent to $2$
modulo $3$).
In this case the group $\Aut(S)$ contains an element of order $n$ as well.

\item[(ii)]
Let $G$ be a finite group.
There exists a field $\KK$ of characteristic zero and
a non-trivial Severi--Brauer surface over $\KK$ such that the group
$\Aut(S)$  contains a subgroup isomorphic to $G$ if and only if
there is a positive integer $n$
divisible only by primes congruent to $1$
modulo $3$ such that $G$
is isomorphic either to $\mumu_n$, or to $\mumu_{3n}$, or to a balanced
semidirect product $\mumu_n\rtimes\mumu_3$, or to
the direct product $\mumu_3\times(\mumu_n\rtimes\mumu_3)$,
where the semidirect product is balanced.

\item[(iii)]
Let $G$ be a finite group.
There exists a field $\KK$ of characteristic zero and
a non-trivial Severi--Brauer surface over $\KK$ such that the group
$\Bir(S)$  contains a subgroup isomorphic to $G$ if and only if
either $G$ is one of the groups listed in assertion~(ii),
or~\mbox{$G\cong\mumu_3^3$.}
\end{itemize}
\end{theorem}

One can reformulate Theorem~\ref{theorem:main}
as a description of the sets of all finite orders of elements
and all finite subgroups of the groups $\Aut(S)$ and $\Bir(S)$,
where $S$ varies in the set $\cS$ of all non-trivial Severi--Brauer surfaces
over all fields of characteristic zero.
Namely, denote
\begin{equation*}
\begin{aligned}
\mathcal{AO}=&\left\{ n\mid n\in\ZZ_{\ge 1}, \text{\ and there exists\ }
S\in\cS \right.\\
&\left.
\text{such that\ } \Aut(S)
\text{\ contains an element of order\ } n\right\},\\
\mathcal{BO}=&\left\{ n\mid n\in\ZZ_{\ge 1}, \text{\ and there exists\ }
S\in\cS\right. \\
&\left.
\text{such that\ } \Bir(S)
\text{\ contains an element of order\ } n\right\},\\
\mathcal{AG}=&\left\{ G\mid |G|<\infty, \text{\ and there exists\ }
S\in\cS\right. \\
&\left.
\text{such that\ } \Aut(S)
\text{\ contains a subgroup isomorphic to\ } G\right\},\\
\mathcal{BG}=&\left\{ G\mid |G|<\infty, \text{\ and there exists\ }
S\in\cS\right. \\
&\left.
\text{such that\ } \Bir(S)
\text{\ contains a subgroup isomorphic to\ } G\right\}.
\end{aligned}
\end{equation*}
Also, denote by $\mathcal{N}$ the set of all positive integers
of the form~\mbox{$n=\prod p_i^{r_i}$},
where $p_i$ are primes congruent to $1$
modulo $3$, and denote
$$
\mathcal{B}=\left\{G \mid G\cong\mumu_n\rtimes\mumu_3
\text{\ is a balanced semidirect product for some\ }
n\in\mathcal{N}\right\}.
$$
In these notation, Theorem~\ref{theorem:main}
claims that
\begin{equation*}
\begin{aligned}
&\mathcal{AO}=\mathcal{BO}=\mathcal{N}\cup
\left\{3n\mid n\in \mathcal{N}\right\},\\
&\mathcal{AG}=\left\{\mumu_n\mid n\in\mathcal{N}\right\}\cup
\left\{\mumu_{3n}\mid n\in\mathcal{N}\right\}\cup
\mathcal{B}\cup\left\{\mumu_3\times G\mid G\in\mathcal{B}\right\},\\
&\mathcal{BG}=\mathcal{AG}\cup\left\{\mumu_3^3\right\}.
\end{aligned}
\end{equation*}

As we will see in Remark~\ref{remark:always-3},
every Severi--Brauer surface over a field of characteristic zero
has an automorphism of order $3$. However, over some fields
birational automorphism groups of non-trivial Severi--Brauer
surfaces may be rather poor in finite subgroups.
For instance, in the case of the field $\QQ$ of rational numbers,
Theorem~\ref{theorem:main} implies the following simpler classification.

\begin{corollary}
\label{corollary:Q}
Let $S$ be a non-trivial Severi--Brauer surface over~$\QQ$, and let
$G$ be a finite subgroup of $\Bir(S)$.
Then $G$ is a subgroup of $\mumu_3^3$. In particular, $G$ is abelian.
\end{corollary}

Using the terminology of \cite[Definition~1]{Popov14},
we see from Corollary~\ref{corollary:Q}
that the Jordan constant of the birational automorphism group
of every non-trivial Severi--Brauer surface over~$\QQ$ equals $1$,
while for a non-trivial Severi--Brauer surface over an arbitrary field
of characteristic zero it does not exceed~$3$
by Theorem~\ref{theorem:J-3}.
Note that the Jordan constant of the birational automorphism group
of the projective plane over~$\QQ$ equals~$120$
by \cite[Theorem~1.11]{Yasinsky}.
I do not know
if the groups $\mumu_3^2$ or $\mumu_3^3$ can actually be embedded
into birational automorphism groups of some non-trivial Severi--Brauer
surfaces over~$\QQ$.

It would be interesting to obtain a complete classification
of finite groups acting on each non-trivial Severi--Brauer surface depending
on the arithmetic properties of the base field and the correspondning
central simple algebra,
similarly to what was done for conics in~\cite{Garcia-Armas},
cf.~\cite{Beauville}.
Also, it would be interesting to find out
if there exists an (infinite) extension~$\KK'$ of
$\QQ$ and a non-trivial Severi--Brauer surface $S'$
over $\KK'$ such that the group~\mbox{$\Aut(S')$} contains all the groups
listed in Theorem~\ref{theorem:main}(ii), and
$\Bir(S')$ contains all the groups
listed in Theorem~\ref{theorem:main}(iii).

\smallskip
The plan of the paper is as follows.
In~\S\ref{section:semi-direct}  we collect some auxiliary
assertions concerning semidirect
products of cyclic groups.
In~\S\ref{section:preliminaries}
we gather auxiliary facts concerning Severi--Brauer surfaces
and Galois theory.
In~\S\ref{section:Gp}
we construct examples of finite groups acting on non-trivial
Severi--Brauer surfaces.
In~\S\ref{section:prime}
we describe the possible finite orders
of automorphisms of non-trivial Severi--Brauer surfaces.
In~\S\ref{section:subgroups}
we classify finite subgroups of automorphism groups
of non-trivial Severi--Brauer surfaces
and prove Theorem~\ref{theorem:main}.
In~\S\ref{section:Q}
we prove Corollary~\ref{corollary:Q}.

\smallskip
\textbf{Notation.}
Given a field $\KK$, we denote by $\bar{\KK}$ its algebraic closure.
For a variety $X$ defined over $\KK$, we denote by $X_{\bar{\KK}}$
its extension of scalars to~$\bar{\KK}$.

\smallskip
\textbf{Acknowledgements.}
I am grateful to S.\,Gorchinskiy, L.\,Rybnikov,
A.\,Trepalin, and V.\,Vologodsky for useful discussions.
I am also grateful to the referee for 
helpful comments, and especially for Remark~\ref{remark:referee}. 
Special thanks go to D.\,Osipov who spotted a gap in a preliminary version
of the paper and suggested several improvements of the exposition.
I was partially supported by
the HSE University Basic Research Program,
Russian Academic Excellence Project~\mbox{``5-100''},
and by the Foundation for the
Advancement of Theoretical Physics and Mathematics ``BASIS''.

\section{Semidirect products}
\label{section:semi-direct}

In this section we collect some elementary assertions concerning semidirect
products of certain finite cyclic groups.

Let $n$ be a positive integer. Write
\begin{equation*}
n=\prod\limits_{i=1}^k p_i^{r_i},
\end{equation*}
where $p_i$ are pairwise different prime numbers.
There is a canonical isomorphism
$$
\mumu_n\cong\prod\limits_{i=1}^k \mumu_n(p_i),
$$
where $\mumu_n(p_i)\cong\mumu_{p_i^{r_i}}$ is the $p_i$-Sylow subgroup of
$\mumu_n$.
For the multiplicative group $\mumu_n^*$ of
residues modulo~$n$ that are coprime to~$n$,
one has a canonical isomorphism
\begin{equation*}
\mumu_n^*\cong\prod\limits_{i=1}^k \mumu_n(p_i)^*,
\end{equation*}
where $\mumu_n(p_i)^*\cong \mumu_{p_i^{r_i}}^*$.
If $p_i>2$, we have
$$
\mumu_n(p_i)^*\cong \mumu_{p_i^{r_i-1}(p_i-1)},
$$
see for instance \cite[\S\,II.2]{Lang}.

Recall that the group $\Aut(\mumu_n)$ is canonically identified
with $\mumu_n^*$. Thus, for every homomorphism
$\chi\colon \mumu_3\to\mumu_n^*$ we can construct a semidirect product
$G\cong\mumu_n\rtimes\mumu_3$. Vice versa, every semidirect product
corresponds to some homomorphism $\chi$.

\begin{definition}\label{definition:balanced}
Suppose that $n$ is divisible only by primes congruent to $1$
modulo $3$, and
let~\mbox{$\chi\colon \mumu_3\to\mumu_n^*$} be a homomorphism.
We say that $\chi$ is \emph{balanced}
if its composition with each of the projections $\mumu_n^*\to\mumu_n(p_i)^*$
is an embedding.
We say that a semidirect product~$G$
of $\mumu_n$ and $\mumu_3$ corresponding to the homomorphism $\chi$
is \emph{balanced} if~$\chi$ is balanced.
\end{definition}

\begin{example}
Let $n=1$. Then the set of prime divisors of $n$ is empty, and thus
Definition~\ref{definition:balanced} does not impose any condition
on a semidirect product. Thus,
the semidirect product $\mumu_3\cong\mumu_1\rtimes\mumu_3$
is balanced.
\end{example}

\begin{example}
Let $p$ be a prime number congruent to $1$ modulo~$3$.
Then the group $G_p$ described in Theorem~\ref{theorem:attained-p}
is a balanced semidirect product of $\mumu_p$ and $\mumu_3$.
Moreover, this group is the unique balanced semidirect product of
$\mumu_p$ and $\mumu_3$.
\end{example}

\begin{example}\label{example:not-balanced}
Let $p_1$ and $p_2$ be distinct prime numbers congruent to $1$ modulo $3$,
and let $G_{p_1}\cong\mumu_{p_1}\rtimes\mumu_3$ be the balanced
semidirect product.
Then the group $G_{p_1}\times \mumu_{p_2}$
is isomorphic to a semidirect product $\mumu_{p_1p_2}\rtimes\mumu_3$
which is not balanced. Indeed, the image of the corresponding homomorphism
$$
\chi\colon \mumu_3\to\mumu^*_{p_1p_2}\cong \mumu_{p_1-1}\times\mumu_{p_2-1}
$$
is contained in the subgroup $\mumu_{p_1p_2}(p_1)^*\cong\mumu_{p_1-1}$
of~$\mumu^*_{p_1p_2}$.
\end{example}

We point out that for a given positive integer $n$
a balanced semidirect product of
order~$3n$ may be not unique.

\begin{example}
Let $p_1$ and $p_2$ be distinct prime numbers congruent to $1$ modulo $3$,
and let $n=p_1p_2$. Then $\mumu_n^*\cong  \mumu_{p_1-1}\times\mumu_{p_2-1}$,
and each of the (cyclic) groups $\mumu_n(p_i)^*\cong\mumu_{p_i-1}$
contains a unique subgroup of order~$3$. Let $\delta_i$, $i=1,2$, be generators of these subgroups.
Set $d_1=\delta_1\delta_2$ and $d_2=\delta_1\delta_2^{-1}$.
Let $\chi_1, \chi_2\colon\mumu_3\to\mumu_n^*$ be the homomorphisms that send
a generator of $\mumu_3$ to $d_1$ and~$d_2$, respectively. Construct the groups
$G_1$ and $G_2$ as semidirect products of $\mumu_n$ and~$\mumu_3$ corresponding to the homomorphisms $\chi_1$ and $\chi_2$, respectively. We claim that
$G_1$ is not isomorphic to $G_2$. Indeed, while it is easy to construct
an automorphism $\zeta$ of $\mumu_n^*$
such that $\chi_2=\zeta\circ\chi_1$, there does not exist an \emph{inner}
automorphism~$\zeta$ of $\mumu_n^*$ like this
(that is, the automorphism~$\zeta$ cannot be obtained as a conjugation with an element
of~$\mumu_n^*$), because the group $\mumu_n^*$ is abelian.
Therefore, the groups $G_1$ and $G_2$ are not isomorphic to each
other by~\mbox{\cite[Theorem~5.1]{BE99}}.
\end{example}

The next result shows that all non-balanced semidirect products
have a structure similar to what we see in Example~\ref{example:not-balanced}.

\begin{lemma}\label{lemma:non-balanced-structure}
Let $n$ be a positive integer divisible only by primes congruent to $1$
modulo~$3$,
let $\chi\colon \mumu_3\to\mumu_n^*$ be a homomorphism,
and let $G$ be a
semidirect product of $\mumu_n$ and~$\mumu_3$ corresponding to $\chi$.
Then $\chi$ is balanced if and only if the only element of~$\mumu_n$
that commutes with
the subgroup $\mumu_3$ is trivial. Moreover,
if $\chi$ is not balanced, then there exists
an isomorphism
\begin{equation}\label{eq:non-balanced-structure}
G\cong(\mumu_{n_1}\rtimes\mumu_3)\times\mumu_{n_2},
\end{equation}
where $n_1n_2=n$ and $n_2>1$.
\end{lemma}

\begin{proof}
Suppose that $\chi$ is not balanced. Then there exists a prime divisor
$p$ of $n$ such that the composition of $\chi$ with the projection
$\mumu_n^*\to\mumu_n(p)^*$ is trivial. Hence, the group $\mumu_3$ commutes with the $p$-Sylow subgroup $\mumu_n(p)\cong p^r$ of $\mumu_n$.
Therefore, we get a required isomorphism
$G\cong(\mumu_{n_1}\rtimes\mumu_3)\times\mumu_{n_2}$, where
$n_1=\frac{n}{p^r}$ and $n_2=p^r$.

Conversely, suppose that $\chi$ is balanced. Denote by $d\in\mumu_n^*$
the image of a generator of $\mumu_3$ with respect to $\chi$, so that
$d^3\equiv 1\pmod n$.
Assume that there exists a non-trivial element $t\in\mumu_n$
that commutes with $\mumu_3$.
Considering~$\mumu_n$ as the group of residues modulo~$n$,
we can write
$$
td\equiv t\pmod n.
$$
Since~$\mumu_n$ is isomorphic to a direct product of its Sylow subgroups,
we conclude that there is a prime divisor $p$ of $n$
such that the image $t_p$ of $t$ in the $p$-Sylow subgroup
$\mumu_n(p)\cong\mumu_{p^r}$ of~$\mumu_n$ is non-trivial.
We have
\begin{equation}\label{eq:tp}
t_pd\equiv t_p\pmod {p^r}.
\end{equation}
Also, we know that
$d\not\equiv 1\pmod {p^r}$ because the homomorphism $\chi$ is balanced,
and~\mbox{$d^3\equiv 1\pmod {p^r}$}.

We see from~\eqref{eq:tp} that $d\equiv 1\pmod p$. Let $s$ be the maximal
integer such that~\mbox{$d\equiv 1\pmod {p^s}$}; then $1\le s<r$.
Write
$$
d\equiv cp^{s}+1\pmod {p^{s+1}},
$$
where $c\not\equiv 0\pmod p$.
One has
$$
1\equiv d^3\equiv 3cp^{s}+1 \pmod {p^{s+1}}.
$$
Since the numbers $3$ and $c$ are not divisible by $p$,
this gives a contradiction.
\end{proof}

\begin{remark}\label{remark:improved-non-balanced}
In the notation of Lemma~\ref{lemma:non-balanced-structure}, one can choose
$n_1$ and $n_2$ so that the semidirect product
$\mumu_{n_1}\rtimes\mumu_3$ in~\eqref{eq:non-balanced-structure}
is balanced.
\end{remark}

We conclude this section by two simple
general observations concerning semidirect
products.

\begin{lemma}\label{lemma:semi-direct-product-relations}
Let $n$ and $m$ be positive integers, and let $d$ be an integer such
that~\mbox{$d^m\equiv 1\pmod n$}.
Let $H$ be a group generated by an element $x$ of order $n$
and an element $y$ of order $m$ subject to relation
\begin{equation}\label{eq:xyyx}
xy=yx^d.
\end{equation}
Suppose that the cyclic subgroups generated by $x$ and $y$ have trivial
intersection.
Then~\mbox{$H\cong \mumu_n\rtimes\mumu_m$}, and the semidirect product
structure corresponds to the homomorphism~\mbox{$\mumu_m\to\mumu_n^*$}
sending a generator of $\mumu_m$ to $d\in\mumu_n^*$.
\end{lemma}

\begin{proof}
We see from relation~\eqref{eq:xyyx}
that every element of $H$ can be written in the form~$y^rx^s$
for some positive integers $r$ and $s$.
Therefore, $H$ is a quotient of the
semidirect product~\mbox{$\tilde{H}\cong \mumu_n\rtimes\mumu_m$}
corresponding to the homomorphism
$\mumu_m\to\mumu_n^*$ that sends a generator of~$\mumu_m$ to $d\in\mumu_n^*$.
On the other hand, if for some $r_1$, $r_2$, $s_1$, and $s_2$
one has~\mbox{$y^{r_1}x^{s_1}=y^{r_2}x^{s_2}$}, then
$y^{r_1-r_2}=x^{s_2-s_1}$.
Since the cyclic subgroups generated by $x$ and $y$ have trivial
intersection, this implies that $r_1\equiv r_2\pmod m$
and $s_1\equiv s_2\pmod n$. Hence~\mbox{$|H|=mn=|\tilde{H}|$},
which means that $H\cong\tilde{H}$.
\end{proof}

\begin{lemma}\label{lemma:is-semidirect}
Let $G$ be a finite group that does not contain
elements of order~$9$, and let~$H$ be a normal subgroup of index~$3$ in~$G$.
Then $G\cong H\rtimes \mumu_3$.
\end{lemma}

\begin{proof}
Choose a preimage $g$ of
a generator of the group $G/H\cong\mumu_3$ with respect to the projection
$\psi\colon G\to G/H$. The order of $g$ is divisible by $3$, but is not
divisible by $9$. Thus it equals $3m$ for some positive integer $m$
not divisible by $3$. Hence the element $g'=g^m$ has order $3$,
and $\psi(g')$ generates $G/H$.
This means that the subgroup $F\cong\mumu_3$ in $G$ generated by $g'$
has trivial intersection with $H$, and thus $G$ is a semidirect
product of $H$ and $F$.
\end{proof}

We remark that one cannot drop the assumption about the absence of elements
of order~$9$ in Lemma~\ref{lemma:is-semidirect}. Indeed, the group
$\mumu_9$ contains a normal subgroup $\mumu_3$ of index~$3$, but
it is not isomorphic to a (semi)direct product of two copies of $\mumu_3$.
Note also that under a stronger assumption that $H$ does not contain elements
of order $3$ the assertion of Lemma~\ref{lemma:is-semidirect} becomes
a particular case of a much more general Schur--Zassenhaus theorem,
see e.g.~\mbox{\cite[Theorem~3.5]{Isaacs}}.

\section{Central simple algebras}
\label{section:preliminaries}

In this section we gather auxiliary facts concerning Severi--Brauer surfaces,
and also some assertions from Galois theory.

We refer the reader to \cite{Artin} and~\cite{Kollar-SB} for the basic theory of Severi--Brauer surfaces and higher-dimensional Severi--Brauer varieties.
The foundational fact of this theory is that Severi--Brauer varieties of dimension $m$ over a field $\KK$  are in one-to-one correspondence with central simple algebras of dimension~$(m+1)^2$ over $\KK$,
with $\PP^m$ corresponding to the algebra of $(m+1)\times (m+1)$-matrices.
A Severi--Brauer variety over $\KK$
is trivial if and only if it has a $\KK$-point.

Given an algebra $A$, we denote by $A^*$ the multiplicative group of its invertible elements.
The following characterization of the automorphism group of a Severi--Brauer variety is well-known,
see Theorem~E on page~266
of~\cite{Chatelet}, or~\mbox{\cite[Lemma~4.1]{SV}}.

\begin{lemma}\label{lemma:SB-Aut}
Let $X$ be a Severi--Brauer variety over a field $\KK$ corresponding to a central simple algebra $A$.
Then $\Aut(X)\cong A^*(\KK)/\KK^*$.
\end{lemma}

Let $\LL/\KK$ be a Galois extension with Galois group
isomorphic to $\mumu_3$. Choose an element~\mbox{$a\in\KK^*$}
and a generator $\sigma$ of $\Gal\left(\LL/\KK\right)$.
Then one can associate with $\sigma$ and $a$
a \emph{cyclic algebra} $A$, which is a central simple algebra over $\KK$,
see \cite[\S2.5]{GilleSzamuely}
or~\mbox{\cite[Exercise~3.1.6]{GorchinskyShramov}}.
Explicitly, $A$ is generated over $\KK$ by $\LL$ and an element $\alpha$
subject to relations~\mbox{$\alpha^3=a$} and
\begin{equation*}
\lambda \alpha=\alpha\sigma(\lambda), \quad \lambda\in\LL.
\end{equation*}

\begin{lemma}[{see e.g. \cite[Exercise~3.1.6(i)]{GorchinskyShramov}}]
\label{lemma:not-matrix}
Suppose that in the above notation the element~$a$ is not contained in the image of the Galois norm $N_{\LL/\KK}$.
Then $A$ is not isomorphic to a matrix algebra.
\end{lemma}

It appears that the above construction is responsible for all possible
central simple algebras of dimension $9$ over an arbitrary field $\KK$
(with a trivial exception of the situation when $\KK$ has
no cyclic extensions of degree $3$, in which case there are no
central simple algebras of dimension $9$ over $\KK$ apart from the matrix
algebra anyway). Namely, the following result
can be found in~\cite{Wedderburn}
or~\mbox{\cite[Chapter~7, Exercise~9]{GilleSzamuely}}.

\begin{lemma}\label{lemma:all-cyclic}
Let $A$ be a central simple algebra of dimension $9$ over a field $\KK$.
Suppose that $A$ is not isomorphic to the matrix algebra.
Then~$A$ is a cyclic algebra constructed from some cyclic Galois
extension~$\LL/\KK$ and some element~\mbox{$a\in\KK^*$}.
\end{lemma}

The next fact is immediately implied by the theorem of Wedderburn  about the structure of central simple algebras, see~\cite[Theorem~2.1.3]{GilleSzamuely}.

\begin{lemma}\label{lemma:matrix-or-division}
Let $A$ be a central simple algebra of dimension $q^2$ over a field $\KK$,
where~$q$ is a prime number.
Then $A$ is either the matrix algebra, or a division algebra.
\end{lemma}

We will also need the following facts from Galois theory.

\begin{theorem}[{\cite[\S\,III.1]{CasselsFrolich}}]
\label{theorem:cyclotomic}
Let $n$ be a positive integer, and let $\xi$ be
a primitive $n$-th root of unity.
Then $\QQ(\xi)/\QQ$ is a Galois extension, and
$$
\Gal\left(\QQ(\xi)/\QQ\right)\cong\mumu_n^*.
$$
\end{theorem}

\begin{lemma}\label{lemma:different-xis}
Let $\xi_1$ and $\xi_2$ be primitive roots of unity of
degrees $n_1$ and $n_2$, respectively.
Suppose that $n_1$ and $n_2$ are coprime, and set
$\MM=\QQ(\xi_1)$. Then $\MM(\xi_2)/\MM$ is a Galois extension, and
$$
\Gal\left(\MM(\xi_2)/\MM\right)\cong\mumu_{n_2}^*.
$$
\end{lemma}

\begin{proof}
Note that $\QQ(\xi_1)\cap\QQ(\xi_2)=\QQ$, see for instance
\cite[\S\,VIII.3]{Lang}.
Also, we know from Theorem~\ref{theorem:cyclotomic} that
$\QQ(\xi_2)/\QQ$ is a Galois extension, and
$$
\Gal\left(\QQ(\xi_2)/\QQ\right)\cong\mumu_{n_2}^*.
$$
Thus the assertion follows from \cite[Theorem~VIII.1.4]{Lang}.
\end{proof}

The next result is well known to experts; it follows for instance from \cite[Theorem~1(b)]{Stern}.

\begin{lemma}
\label{lemma:norms}
Let $\LL/\KK$ be a finite Galois extension of number fields.
Then the Galois norm map $N_{\LL/\KK}\colon \LL^*\to\KK^*$ is not surjective.
\end{lemma}

\section{Construction of subgroups}
\label{section:Gp}

In this section we construct examples of finite groups acting on non-trivial
Severi--Brauer surfaces.
It is easy to produce a non-trivial Severi--Brauer surface
with an action of the group~$\mumu_3^2$.

\begin{example}[{see \cite[Example~4.7]{Shramov-SB}}]
\label{example:cyclic-algebra}
Let $\KK$ be a field of characteristic different from~$3$
that contains a non-trivial
cubic root of unity $\omega$. Let $a$ and $b$ be elements of $\KK$
such that~$b$ is not
a cube in $\KK$, and~$a$ is not contained in
the image of the Galois norm for the
field extension~$\LL/\KK$, where
$$
\LL=\KK\left(\sqrt[3]{b}\right).
$$
If $\KK$ is a number
field, then an element $b$ like this always exists for obvious reasons, and
an element $a$ exists by Lemma~\ref{lemma:norms}.
Consider the algebra~$A$ over~$\KK$ generated by variables~$u$ and~$v$
subject to relations
$$
u^3=a,\quad v^3=b,\quad uv=\omega vu.
$$
Then $A$ is a central simple algebra, and moreover a division algebra;
actually, one can construct it as a cyclic algebra over $\KK$
associated with (a generator of the Galois group~of)
the cyclic extension $\LL/\KK$ and the element~$a$.
Thus $A$ corresponds to a non-trivial Severi--Brauer
surface $S$.
One can see that $u$ and $v$
generate a non-abelian subgroup of order $27$ and exponent $3$
in $A^*$ (sometimes called the Heisenberg group).
Their images in~\mbox{$A^*/\KK^*$} generate a group isomorphic to $\mumu_3^2$,
which is embedded as a subgroup into $\Aut(S)$ by Lemma~\ref{lemma:SB-Aut}.
\end{example}

\begin{remark}\label{remark:always-3}
Let $S$ be a Severi--Brauer surface over a field $\KK$.
Then~$S$ has an automorphism of order $3$. Indeed, if $S\cong\PP^2$,
then this is obvious. Otherwise it follows from Lemma~\ref{lemma:all-cyclic}
that $S$ corresponds to some cyclic algebra. In the notation
of~\S\ref{section:preliminaries}, the algebra $A$ contains the element
$\alpha\not\in\KK$
such that $\alpha^3\in\KK$. Thus $\alpha$ defines an automorphism
of $S$ of order $3$ by Lemma~\ref{lemma:SB-Aut}.
\end{remark}

The proof of the next result
is similar to that of Theorem~\ref{theorem:attained-p}, see
\cite{Shramov-SB-short}.

\begin{lemma}\label{lemma:attained-n}
Let $n>1$ be an integer divisible only by primes congruent to $1$
modulo~$3$. Let $G\cong\mumu_n\rtimes\mumu_3$ be
a balanced semidirect product.
Then
there exists a number field $\KK$ and a non-trivial Severi--Brauer
surface $S$ over $\KK$ such that
the group $\Aut(S)$ contains~$G$.
\end{lemma}

\begin{proof}
Let $\xi$ denote a primitive root of unity
of degree~$n$, and let $\LL=\QQ(\xi)$.
Then~\mbox{$\LL/\QQ$} is a Galois extension with
the Galois group isomorphic to $\mumu_n^*$ by Theorem~\ref{theorem:cyclotomic}.
Let $\chi\colon\mumu_3\to \mumu_n^*$
be the homomorphism giving rise to the semidirect product structure on $G$.
Consider $\chi$ as a homomorphism
$$
\chi\colon \mumu_{3}\to\Gal\left(\LL/\QQ\right),
$$
and let
$\KK\subset\LL$ be the subfield of invariants of
the subgroup~\mbox{$\chi(\mumu_3)\subset \Gal\left(\LL/\QQ\right)$}.
Then~$\LL/\KK$ is a Galois extension with
the Galois group~\mbox{$\Gal\left(\LL/\KK\right)$}
isomorphic to~$\mumu_3$.
Let $\sigma$ be a generator of $\Gal\left(\LL/\KK\right)$.
Since $\chi$ is balanced,
we conclude from Lemma~\ref{lemma:non-balanced-structure} that
for all $1\le m<n$ one has $\sigma(\xi^m)\neq\xi^m$;
thus~\mbox{$\xi^m\not\in\KK$}. Explicitly, $\sigma$
sends $\xi\in\LL$ to $\xi^d$,
where~\mbox{$d\in\mumu_n^*$} is the image of a generator
of~$\mumu_3$ under the embedding $\chi\colon\mumu_3\to\mumu_n^*$.
In particular, one has $d^3\equiv 1\pmod n$.

By Lemma~\ref{lemma:norms} there exists an element $a\in\KK$ such that $a$
is not contained in the image of
the Galois norm of the field extension $\LL/\KK$.
Consider the cyclic algebra $A$ over~$\KK$ associated with $\sigma$ and $a$.
We know that $A$ is a central simple algebra. However, $A$ is not a matrix algebra by Lemma~\ref{lemma:not-matrix}.
Since the dimension of $A$ over $\KK$ equals $9$, we conclude that
$A$ is a division algebra by Lemma~\ref{lemma:matrix-or-division}.
The algebra $A$ contains the root of unity~\mbox{$\xi\in\LL$}
and an element~$\alpha$ such that $\alpha^3=a$ and
relation
\begin{equation}\label{eq:relation}
\xi \alpha=\alpha\xi^d
\end{equation}
holds.

Denote by $\hat{\alpha}$ and $\hat{\xi}$ the images in
the quotient group $A^*/\KK^*$ of the elements $\alpha$ and~$\xi$,
respectively, and denote by
$\hat{G}$ the subgroup of $A^*/\KK^*$
generated by $\hat{\alpha}$ and $\hat{\xi}$.
Since $\alpha^3\in\KK$ and $\alpha\not\in \KK$, the order
of $\hat{\alpha}$ in $A^*/\KK^*$ equals~$3$ as well.
Similarly, one has $\xi^n=1\in\KK$, and~\mbox{$\xi^m\not\in\KK$}
for all $1\le m<n$.
Thus the order of  $\hat{\xi}$ in $A^*/\KK^*$  equals $n$.
Since $n$ is coprime to $3$, the cyclic groups generated
by $\hat{\xi}$ and $\hat{\alpha}$ have trivial intersection.
Furthermore, relation~\eqref{eq:relation} implies
$$
\hat{\xi}\hat{\alpha}=\hat{\alpha}\hat{\xi}^d.
$$
Hence $\hat{G}$ is the semidirect product of $\mumu_n$ and $\mumu_3$
corresponding to $\chi$ by Lemma~\ref{lemma:semi-direct-product-relations},
so that $\hat{G}\cong G$.

We see that the elements $\hat{\xi}$ and $\hat{\alpha}$
generate a subgroup isomorphic to $G$ in $A^*/\KK^*$.
On the other hand, for a Severi--Brauer surface $S$ corresponding to
$A$ one has $\Aut(S)\cong A^*/\KK^*$ by Lemma~\ref{lemma:SB-Aut}.
\end{proof}

\begin{remark}\label{remark:not-attained-non-balanced}
If in the notation of the proof of Lemma~\ref{lemma:attained-n}
we choose $\chi\colon \mumu_3\to\Gal\left(\LL/\QQ\right)$
to be a non-balanced homomorphism, and set $\KK$ to be the field of invariants
of the group~\mbox{$\chi(\mumu_3)$},
then for some $1\le m<n$ the element $\xi^m$ is invariant under the action
of the Galois group $\Gal\left(\LL/\KK\right)$. This means that
$\xi^m$ is an element of $\KK$. Therefore, while
the subgroup of $A^*$ generated by $\xi$ and $\alpha$ has a quotient
isomorphic to
the non-balanced semidirect product $\mumu_n\rtimes\mumu_3$ corresponding to
$\chi$, its image in $\Aut(S)\cong A^*/\KK^*$ is isomorphic to a (balanced)
semidirect product $\mumu_{n'}\rtimes\mumu_3$ for some~\mbox{$n'<n$},
cf. Remark~\ref{remark:improved-non-balanced}.
\end{remark}

The following result is proved by a blend of constructions
used in Example~\ref{example:cyclic-algebra}
and Lemma~\ref{lemma:attained-n}.

\begin{lemma}\label{lemma:attained-3n}
Let $n\ge 1$
be an integer divisible only by primes congruent to $1$
modulo~$3$. Let $G'\cong\mumu_{n}\rtimes\mumu_3$ be
a balanced semidirect product, and let $G\cong\mumu_3\times G'$.
Then there exists a number field $\KK$ and a non-trivial Severi--Brauer
surface $S$ over $\KK$ such that
the group $\Aut(S)$ contains~$G$.
\end{lemma}

\begin{proof}
Let $\omega$ and $\xi$ denote primitive roots of unity
of degrees $3$ and~$n$, respectively; let~$\tau$ be a cubic root
of~$2$. Set $\MM=\QQ(\omega)$ and $\LL=\MM(\xi, \tau)$.
We observe that
$$
\MM(\tau)\cap\MM(\xi)=\MM.
$$
Indeed, this intersection contains $\MM$ and is contained in
$\MM(\tau)$. The degree of the extension~\mbox{$\MM(\tau)/\MM$}
equals $3$, so we only need to check that $\MM(\tau)$ is not contained in
$\MM(\xi)$. On the other hand, if  $\MM(\tau)\subset\MM(\xi)$,
then $\QQ(\tau)$ is a subfield of $\MM(\xi)$ as well. Since
$$
\MM(\xi)=\QQ(\omega,\xi)=\QQ(\omega\xi),
$$
we know from Theorem~\ref{theorem:cyclotomic} that the Galois group
$\Gal\left(\MM(\xi)/\QQ\right)$ is cyclic. Hence every
subfield of $\MM(\xi)$ is normal over $\QQ$, which is not the case
for~$\QQ(\tau)$.

Note that $\MM(\xi)/\MM$
is a Galois extension with Galois group isomorphic to
$\mumu_n^*$ by Lemma~\ref{lemma:different-xis},
and $\MM(\tau)/\MM$ is a Galois
extension with Galois group isomorphic to $\mumu_3$.
Since the extension $\LL/\MM$ can be represented as the composite of
the extensions $\MM(\tau)/\MM$ and $\MM(\xi)/\MM$,
we conclude that
$$
\Gal\left(\LL/\MM\right)\cong
\Gal\left(\MM(\tau)/\MM\right)\times\Gal\left(\MM(\xi)/\MM\right)\cong
\mumu_3\times\mumu_n^*,
$$
see for instance \cite[Theorem~VIII.1.5]{Lang}.

Let $\chi\colon \mumu_{3}\to\Gal\left(\LL/\MM\right)$ be the homomorphism
such that its
composition with the projection
$\Gal\left(\LL/\MM\right)\to\Gal\left(\MM(\tau)/\MM\right)$
is an isomorphism $\mumu_3\stackrel{\sim}\longrightarrow\mumu_3$,
and the composition with the projection
$\Gal\left(\LL/\MM\right)\to\Gal\left(\MM(\xi)/\MM\right)$
is the homomorphism
$$
\chi'\colon \mumu_3\to\mumu_n^*
$$
giving rise to the semidirect product structure on $G'$.
Let $\KK\subset\LL$ be the subfield of invariants of
the subgroup~\mbox{$\chi(\mumu_3)\subset \Gal\left(\LL/\MM\right)$}.
Then $\LL/\KK$ is a Galois extension such that
$\Gal\left(\LL/\KK\right)\cong\mumu_3$.
A generator $\sigma$ of $\Gal\left(\LL/\KK\right)$
sends $\xi\in\LL$ to $\xi^d$,
where $d$
is some integer such that~\mbox{$d^3\equiv 1\pmod n$},
and sends $\tau\in\LL$ to~$\omega\tau$.
Since the homomorphism $\chi'$ is balanced,
we conclude from Lemma~\ref{lemma:non-balanced-structure} that
for all $1\le m<n$ one has $\sigma(\xi^m)\neq\xi^m$;
thus~\mbox{$\xi^m\not\in\KK$}.

By Lemma~\ref{lemma:norms} there exists an element $a\in\KK$ such that $a$
is not contained in the image
of the Galois norm of the field extension $\LL/\KK$.
Consider the cyclic algebra $A$ over~$\KK$ associated with $\sigma$ and $a$.
We know that $A$ is a central simple algebra, and
$A$ is not a matrix algebra by Lemma~\ref{lemma:not-matrix}.
Thus $A$ is a division algebra by Lemma~\ref{lemma:matrix-or-division}.
The algebra~$A$ contains the $n$-th root of unity $\xi$
and an element~$\alpha$ such that $\alpha^3=a$ and
relation~\eqref{eq:relation}.
holds. Furthermore, $A$ contains $\tau$, and
\begin{equation}\label{eq:tau}
\tau\alpha=\omega\alpha\tau.
\end{equation}

Denote by $\hat{\alpha}$, $\hat{\xi}$, and $\hat{\tau}$
the images in
the quotient group $A^*/\KK^*$ of the elements $\alpha$,~$\xi$,
and~$\tau$,
respectively. Denote by
$\hat{G}'$ the subgroup of $A^*/\KK^*$
generated by $\hat{\alpha}$ and $\hat{\xi}$,
and by $\hat{G}$ the subgroup generated by $\hat{\alpha}$, $\hat{\xi}$,
and $\hat{\tau}$.
The order of $\hat{\alpha}$ in $A^*/\KK^*$ equals~$3$.
The order of $\hat{\xi}$ in~\mbox{$A^*/\KK^*$} equals $n$,
and the cyclic groups generated by $\hat{\xi}$ and $\hat{\alpha}$
have trivial intersection.
Using relation~\eqref{eq:relation} and Lemma~\ref{lemma:semi-direct-product-relations}, we see that
$\hat{G}'$ is the semidirect product of $\mumu_n$ and~$\mumu_3$
corresponding to the homomorphism $\chi'$, and so
$\hat{G}'\cong G'$
(cf. the proof of Lemma~\ref{lemma:attained-n}).

Since $\omega\in\KK$, we see from relation~\eqref{eq:tau} that
the group $\hat{G}''$ generated by $\hat{\alpha}$ and $\hat{\tau}$
is isomorphic to $\mumu_3^2$ (cf. Example~\ref{example:cyclic-algebra}).
In particular, $\hat{\tau}$ has order $3$ and commutes with~$\hat{\alpha}$.
Since $\tau$ and $\xi$ are elements of the field $\LL\subset A$, they
commute with each other, and thus~$\hat{\tau}$ commutes with~$\hat{\xi}$.
Hence~$\hat{\tau}$ commutes with the group $\hat{G}'$ generated by
$\hat{\alpha}$ and $\hat{\xi}$.
Since~$n$ is coprime to $3$, the group $\hat{G}''$
has trivial intersection
with the cyclic group generated by~$\hat{\xi}$.
This implies that the cyclic group $\mumu_3$ generated by $\hat{\tau}$
has trivial intersection with the group $\hat{G}'$. Thus we conclude that
$$
\hat{G}\cong\mumu_3\times \hat{G}'\cong\mumu_3\times G'\cong G.
$$

Therefore, the elements $\hat{\xi}$, $\hat{\alpha}$, and $\hat{\tau}$
generate a subgroup isomorphic to $G$ in~\mbox{$A^*/\KK^*$}.
Thus the automorphism group of the Severi--Brauer surface corresponding to
$A$ contains~$G$ by Lemma~\ref{lemma:SB-Aut}.
\end{proof}

Lemmas~\ref{lemma:attained-n} and \ref{lemma:attained-3n}
imply the following.

\begin{corollary}\label{corollary:elements-attained}
Let~\mbox{$n=3^r\prod p_i^{r_i}$} be a positive integer,
where $p_i$ are primes congruent to~$1$
modulo $3$, and $r\le 1$. Then
there exists a number field $\KK$ and a non-trivial Severi--Brauer
surface $S$ over~$\KK$ such that
the group $\Aut(S)$ contains an element of order~$n$.
\end{corollary}

\section{Automorphisms of prime order}
\label{section:prime}

In this section we describe the possible finite orders
of automorphisms of non-trivial Severi--Brauer surfaces.

Corollary~\ref{corollary:elements-attained} shows that for every positive
integer $n$
not divisible by $9$ and
not divisible by primes congruent to $2$
modulo~$3$ there exists a non-trivial Severi--Brauer surface
over a field of characteristic zero
with an automorphism of order $n$.
We complete this picture by showing that these are all
possible finite orders of automorphisms of non-trivial Severi--Brauer surfaces.
Let us start with a general observation.

\begin{lemma}\label{lemma:unique-point}
Let $S$ be a non-trivial Severi--Brauer surface over a field $\KK$ of characteristic zero.
Let $g$ be a non-trivial automorphism of finite order of $S$. Then $g$ has exactly three
fixed points on~$S_{\bar{\KK}}$, and these points are transitively permuted
by the Galois group~\mbox{$\Gal\left(\bar{\KK}/\KK\right)$}.
In particular, $g$ cannot have a unique isolated fixed point
on~$S_{\bar{\KK}}$.
\end{lemma}

\begin{proof}
Since $g$ has finite order, it either has exactly three fixed points on~$S_{\bar{\KK}}$, or
it acts on~$S_{\bar{\KK}}$ with a unique isolated fixed point.
Suppose that the latter is the case.
Since the action of $\Gal\left(\bar{\KK}/\KK\right)$
on $S$ commutes with the action of $g$, the isolated fixed point must be
$\Gal\left(\bar{\KK}/\KK\right)$-invariant,
and thus defined over $\KK$.
This is impossible because the Severi--Brauer surface
$S$ is non-trivial.
Therefore, $g$ has exactly three fixed points on~$S_{\bar{\KK}}$;
since~$S$ has no points over~$\KK$, these three points have to be
transitively permuted
by~\mbox{$\Gal\left(\bar{\KK}/\KK\right)$}.
\end{proof}

\begin{lemma}[{cf. \cite[Lemma~4.2]{Shramov-SB}}]
\label{lemma:3k2}
Let $p\equiv 2\pmod 3$ be a prime number.
Let~$S$ be a Severi--Brauer surface over a field $\KK$ of characteristic zero
such that the group~\mbox{$\Aut(S)$}
contains an element $g$ of order $p$. Then $S\cong\PP^2$.
\end{lemma}

\begin{proof}
Suppose that the Severi--Brauer surface $S$ is non-trivial, and
consider the action of $g$ on $S_{\bar{\KK}}\cong\PP^2_{\bar{\KK}}$.
Let $\tilde{g}$ be a preimage of $g$ under the natural projection
$$
\pi\colon\SL_3\left(\bar{\KK}\right)\to\Aut\left(\PP_{\bar{\KK}}^2\right)
\cong\PGL_3\left(\bar{\KK}\right).
$$
The order of $\tilde{g}$ equals either $p$ or $3p$.
Since $p$ is not divisible by $3$, we can multiply $\tilde{g}$ by
an appropriate scalar matrix so that the order of $\tilde{g}$ equals $p$.
Moreover, the element~$\tilde{g}$ of order $p$ such that
$\pi(\tilde{g})=g$ is unique
in~$\SL_3\left(\bar{\KK}\right)$, and
hence it is defined over the field~$\KK$.

Let $\xi_1$, $\xi_2$, and $\xi_3$ be the eigen-values of $\tilde{g}$;
these are $p$-th roots of unity.
According to Lemma~\ref{lemma:unique-point}, the element $g$
has exactly three fixed points on $\PP^2_{\bar{\KK}}$, and these points
are transitively permuted by the group $\Gal\left(\bar{\KK}/\KK\right)$.
In other words, the roots of unity $\xi_1$, $\xi_2$, and $\xi_3$
are pairwise different and form a $\Gal\left(\bar{\KK}/\KK\right)$-orbit.
Hence the image $\Gamma$ of~\mbox{$\Gal\left(\bar{\KK}/\KK\right)$}
in the automorphism
group $\Aut(\mumu_p)\cong\mumu_{p-1}$ of
the multiplicative group $\mumu_p$ of $p$-th roots of unity has an
orbit of order~$3$. This means that the order of $\Gamma$ is divisible
by $3$. However, the number~\mbox{$p-1$}
is not divisible by $3$ by assumption, which gives
a contradiction.
\end{proof}

The argument from the proof of Lemma~\ref{lemma:3k2}
cannot be applied to an automorphism of a Severi--Brauer surface
whose order is divisible by $3$, since in this case we cannot always
find a preimage of such an automorphism in $\SL_3\left(\bar{\KK}\right)$
defined
over the base field~$\KK$. However, we can study such elements
using a slightly modified argument.

\begin{lemma}
\label{lemma:order-9}
Let~$S$ be a Severi--Brauer surface over a field $\KK$ of characteristic zero
such that the group~\mbox{$\Aut(S)$}
contains an element $g$ of order $9$. Then $S\cong\PP^2$.
\end{lemma}

\begin{proof}
Suppose that the Severi--Brauer surface $S$ is non-trivial.
According to Lemma~\ref{lemma:unique-point}, the element $g$
has exactly three fixed points $P_1$, $P_2$, and $P_3$
on $\PP^2_{\bar{\KK}}$, and these points
are transitively permuted by the group $\Gal\left(\bar{\KK}/\KK\right)$.
Let $T_i\cong\bar{\KK}^2$
be the Zariski tangent space to $\PP^2_{\bar{\KK}}$ at the point
$P_i$. Then $g$ acts faithfully on $T_i$.
Let $\xi_i'$ and $\xi_i''$ be the eigen-values of this action.
Choose some preimage $\tilde{g}$ of $g$ under the projection
$$
\pi\colon\SL_3\left(\bar{\KK}\right)\to\PGL_3\left(\bar{\KK}\right);
$$
we emphasize that $\tilde{g}$ may not be defined over $\KK$.
Let $\xi_1$, $\xi_2$, and $\xi_3$ be the eigen-values of $\tilde{g}$
corresponding to the fixed
points $P_1$, $P_2$, and $P_3$ of $g$, respectively. Relabelling the
numbers $\xi_i'$ and $\xi_i''$ if necessary, we may assume that
$$
\xi_1'=\frac{\xi_2}{\xi_1}, \quad \xi_1''=\frac{\xi_3}{\xi_1}, \quad
\xi_2'=\frac{\xi_3}{\xi_2}, \quad \xi_2''=\frac{\xi_1}{\xi_2}, \quad
\xi_3'=\frac{\xi_1}{\xi_3}, \quad \xi_3''=\frac{\xi_2}{\xi_3}.
$$
Note that $\xi_i'$ and $\xi_i''$ are $9$-th roots of unity, while
$\xi_i$ may not be $9$-th roots of unity but are always $27$-th roots
of unity, because $\tilde{g}^9$ is a scalar matrix
in $\SL_3\left(\bar{\KK}\right)$.

If for some $i\neq j$ we have $\xi_i^3=\xi_j^3$,
then the non-trivial automorphism $g^3$ of $S$ has infinitely many
fixed points on $S_{\bar{\KK}}$, and thus $S\cong\PP^2$
by Lemma~\ref{lemma:unique-point}.
This gives some restrictions on the values of $\xi_i'$ and $\xi_i''$.
For instance,
let $\omega$ be a non-trivial cubic root of unity,
let $\upsilon$ be one of the numbers
$1$, $\omega$, or~$\omega^2$.
If $\xi_1'=\upsilon\xi_2''$, then
$$
\frac{\xi_2}{\xi_1}=\upsilon\frac{\xi_1}{\xi_2},
$$
so that $\upsilon\xi_1^2=\xi_2^2$. Thus $\xi_1^6=\xi_2^6$, and hence
$\xi_1^3=\xi_2^3$ because both $\xi_1$ and $\xi_2$ are $27$-th
roots of unity. Similarly, if $\xi_1'=\upsilon\xi_3''$, then
$$
\frac{\xi_2}{\xi_1}=\upsilon\frac{\xi_2}{\xi_3},
$$
and thus $\upsilon\xi_1=\xi_3$ and $\xi_1^3=\xi_3^3$.
In each of these cases we
see that $S\cong\PP^2$, which contradicts our assumption.

Since the points $P_1$, $P_2$, and $P_3$ are transitively permuted
by the Galois group~\mbox{$\Gal\left(\bar{\KK}/\KK\right)$}, the action of
$\Gal\left(\bar{\KK}/\KK\right)$ transitively permutes the non-ordered pairs
\begin{equation}\label{eq:pairs}
\{\xi_1',\xi_1''\},\quad \{\xi_2',\xi_2''\}, \quad \{\xi_3',\xi_3''\}.
\end{equation}
(Note that at the same time
the collection of eigen-values $\xi_1$, $\xi_2$, $\xi_3$
may be not preserved by the Galois group, since $\tilde{g}$ is not
necessarily defined over the field~$\KK$.)
Thus there exists an element
$\gamma\in\Gal\left(\bar{\KK}/\KK\right)$
such that $\gamma(\xi_1')$ and $\gamma^2(\xi_1')$ are contained in the second
and the third pairs in~\eqref{eq:pairs}, respectively,
and $\gamma^3(\xi_1')=\xi_1'$. Such an element defines an automorphism of order
$3$ of the cyclic subgroup of $\bar{\KK}^*$ generated by $\xi_1'$.
Since~$\xi_1'$ is a $9$-th root of unity, one has
$\gamma(\xi_1')=\upsilon^{-1}\xi_1'$ and
$\gamma^2(\xi_1')=\upsilon^{-2}\xi_1'$,
where $\upsilon$ is one of the numbers~$1$,~$\omega$, or~$\omega^2$.

The above computation shows that
$\gamma(\xi_1')\neq \xi_2''$ and $\gamma^2(\xi_1')\neq\xi_3''$,
so that~\mbox{$\gamma(\xi_1')=\xi_2'$} and~\mbox{$\gamma^2(\xi_1')=\xi_3'$}.
In other words,
we have
$$
\xi_1'=\upsilon\xi_2'=\upsilon^2\xi_3'.
$$
Thus
$$
\frac{\xi_2}{\xi_1}=\upsilon\frac{\xi_3}{\xi_2}=\upsilon^2\frac{\xi_1}{\xi_3}.
$$
Hence
$$
\frac{\xi_2^2}{\xi_1}=\upsilon\xi_3=\frac{\xi_1^2}{\xi_2},
$$
so that $\xi_1^3=\xi_2^3$.
This again means that $S\cong\PP^2$, and gives a contradiction.
\end{proof}

Let us summarize the results of the last two sections.

\begin{corollary}\label{corollary:order}
Let $n$ be a positive integer.
Then there exists a field $\KK$ of characteristic zero, a non-trivial
Severi--Brauer surface $S$ over $\KK$, and an element of order $n$
in the group~\mbox{$\Aut(S)$},
if and only if~\mbox{$n=3^r\prod p_i^{r_i}$},
where $p_i$ are primes congruent to $1$
modulo $3$, and $r\le 1$.
\end{corollary}

\begin{proof}
The ``if'' part of the assertion is given by
Corollary~\ref{corollary:elements-attained}.
For the ``only if'' part, let 
$n$ be the order of some element of $\Aut(S)$. Then every prime divisor of $n$ either equals $3$, or is  
congruent to $1$ modulo $3$ by Lemma~\ref{lemma:3k2}; 
and $n$ is not divisible by $9$ by Lemma~\ref{lemma:order-9}.
\end{proof}

\section{Subgroups}
\label{section:subgroups}

In this section we classify finite subgroups of automorphism groups
of non-trivial Severi--Brauer surfaces and prove Theorem~\ref{theorem:main}.

\begin{lemma}\label{lemma:unique-fixed-point}
Let $\bar{\KK}$ be an algebraically closed field of characteristic zero,
and let~\mbox{$\check{G}\cong\mumu_p^2$},
where $p$ is a prime number,
be a finite subgroup of~\mbox{$\GL_3\left(\bar{\KK}\right)$}.
Suppose that $\check{G}$ does not contain non-trivial scalar matrices.
Then there exists an element of~$\check{G}$
whose fixed point locus on $\PP^2_{\bar{\KK}}$
contains a unique isolated point.
\end{lemma}

\begin{proof}
Since $\check{G}$ is a finite abelian group, we may assume that it consists of diagonal matrices.  
Let $\xi$ be a non-trivial $p$-th root of unity,
and let $\hat{G}$ be the group generated by~$\check{G}$ and the
scalar matrix with diagonal entries equal to~$\xi$.
Obviously, one has $|\hat{G}|=p^3$, and
the image of $\hat{G}$ in $\PGL_3\left(\bar{\KK}\right)$ coincides
with that of $\check{G}$.
Since the total number of diagonal elements of order $p$
in $\GL_3\left(\bar{\KK}\right)$ equals $p^3-1$, all of them must be
contained in~$\hat{G}$. In particular, this applies to the matrix
$$
A=\left(
\begin{array}{ccc}
\xi & 0 & 0\\
0 &  1 & 0\\
0 & 0 & 1
\end{array}
\right).
$$
The fixed locus of the image of $A$ in $\PGL_3\left(\bar{\KK}\right)$
is a union of a line and an isolated fixed point
in~$\PP^2_{\bar{\KK}}$.
\end{proof}

\begin{corollary}\label{corollary:p-squared-on-P2}
Let $p\neq 3$ be a prime number,
and let
$S$ be a Severi--Brauer surface over a field $\KK$ of characteristic zero
such that the group $\Aut(S)$
contains a subgroup $G\cong\mumu_p^2$. Then $S\cong\PP^2$.
In other words, for $p\ge 3$ every finite abelian $p$-group acting on
a non-trivial Severi--Brauer surface over a field of characteristic zero
is cyclic.
\end{corollary}

\begin{proof}
Consider the action of $G$ on $S_{\bar{\KK}}\cong\PP^2_{\bar{\KK}}$.
Consider the natural projection
$$
\pi\colon\SL_3\left(\bar{\KK}\right)\to\Aut\left(\PP^2_{\bar{\KK}}\right)
\cong\PGL_3\left(\bar{\KK}\right).
$$
Set $\tilde{G}=\pi^{-1}(G)$,
so that $\tilde{G}$ is a group of order $3p^2$.
Let $\check{G}$ be the $p$-Sylow subgroup
of~$\tilde{G}$. Then $\pi$
gives an isomorphism between $\check{G}$ and $G$.
Moreover, since $p\neq 3$,
the group $\check{G}$ does not contain non-trivial scalar matrices.
By Lemma~\ref{lemma:unique-fixed-point}
there exists an element $g\in G$
such that the set of fixed points of $g$ on $\PP^2_{\bar{\KK}}$
contains a unique isolated point.
Thus the required assertion follows from Lemma~\ref{lemma:unique-point}.
\end{proof}

\begin{corollary}\label{corollary:unique-fixed-point-3-squared}
Let $S$ be a Severi--Brauer surface over a field $\KK$ of characteristic zero
such that the group $\Aut(S)$
contains a subgroup $G\cong\mumu_3^2$.
Suppose that $G$ has at least one fixed point on
$S_{\bar{\KK}}\cong\PP^2_{\bar{\KK}}$.
Then $S\cong\PP^2$.
\end{corollary}

\begin{proof}
Consider the natural projection
$$
\pi\colon\SL_3\left(\bar{\KK}\right)\to\PGL_3\left(\bar{\KK}\right).
$$
Set $\tilde{G}=\pi^{-1}(G)$,
so that $\tilde{G}$ is a group of order $27$.
Since $G$ has a fixed point on $\PP^2_{\bar{\KK}}$, the
natural three-dimensional representation $V$ of $\tilde{G}$ is reducible.
Recall that the dimension of every irreducible representation of
$\tilde{G}$ either equals $1$, or is divisible by $3$. Hence $V$ splits
into a sum of three one-dimensional representations of $\tilde{G}$,
which means that the group $\tilde{G}$ is abelian.
Choose the elements $\tilde{g}_1$ and $\tilde{g}_2$ in $\tilde{G}$
whose images in $\PGL_3\left(\bar{\KK}\right)$ generate the group $G$.
Then $\tilde{g}_1$ and $\tilde{g}_2$ generate a subgroup
$\check{G}\cong\mumu_3^2$ in $\tilde{G}$ such that $\check{G}$
does not contain non-trivial scalar matrices.
By Lemma~\ref{lemma:unique-fixed-point}
there exists an element $g\in G$
such that the set of fixed points of $g$ on $\PP^2_{\bar{\KK}}$
contains a unique isolated point.
Thus the required assertion follows from Lemma~\ref{lemma:unique-point}.
\end{proof}

\begin{corollary}[{cf. \cite[Lemma~4.1(iii),(iv), Lemma~4.6]{Shramov-SB}}]
\label{corollary:3-squared-on-P2}
Let $S$ be a Severi--Brauer surface over a field $\KK$ of characteristic zero
such that the group $\Aut(S)$
contains a finite $3$-group~$G$. Suppose that $G$ is not
a subgroup of $\mumu_3^2$.
Then $S\cong\PP^2$.
\end{corollary}

\begin{proof}
Suppose that the Severi--Brauer surface $S$ is non-trivial.
Since $G$ is a $3$-group of order at least $27$,
it contains a subgroup $G'$ of order exactly $27$.
Furthermore, the group~$G$, and thus also $G'$,
does not contain elements of order greater than $3$
by Lemma~\ref{lemma:order-9}.
Therefore, if $G'$ is abelian, then~\mbox{$G'\cong\mumu_3^3$};
if $G'$ is non-abelian, then it is the Heisenberg group of order
$27$ and exponent $3$. In both cases $G'$ contains
a subgroup $Z\cong\mumu_3$ that commutes with every element of~$G'$.

We claim that $G'$ contains a subgroup $\mumu_3^2$ acting
on $S_{\bar{\KK}}\cong\PP^2_{\bar{\KK}}$ with a fixed point. Indeed,
let $g_1$ be a generator of $Z$,
let $g_2$ be an element of $G'$ not contained in $Z$, and
let $g_3$ be an element of $G'$ not contained in the subgroup
$\mumu_3^2$ generated by $g_1$ and $g_2$.
According to Lemma~\ref{lemma:unique-point},
the element $g_1$ has exactly three fixed points $P_1$, $P_2$, and $P_3$
on $\PP^2_{\bar{\KK}}$.
Since~$g_2$ commutes with $g_1$, it preserves the set~\mbox{$\{P_1,P_2,P_3\}$}.
Therefore, $g_2$ either fixes each of the points $P_i$,
or permutes them transitively. In the former case the group~$\mumu_3^2$
generated by $g_1$ and $g_2$
acts on $\PP^2_{\bar{\KK}}$ with fixed points.
So, we may assume that~$g_2$ defines a cyclic
permutation of the points~$P_1$,~$P_2$, and~$P_3$. Similarly,
the element~$g_3$ preserves the set~\mbox{$\{P_1,P_2,P_3\}$},
and we may assume that
it defines a cyclic
permutation of~$P_1$,~$P_2$, and~$P_3$. This means that either the element
$g_2g_3$ or the element $g_2g_3^{-1}$ preserves the points
$P_1$, $P_2$, and $P_3$. Together with
$g_1$ this element generates a group $\mumu_3^2$ acting
on~$\PP^2_{\bar{\KK}}$ with a fixed point.

Now Corollary~\ref{corollary:unique-fixed-point-3-squared}
shows that $S\cong\PP^2$.
\end{proof}

While the group $\mumu_3^2$ can act on a non-trivial Severi--Brauer
surface over a field of characteristic zero, it appears that
such an action imposes strong restrictions on other finite groups
acting on this Severi--Brauer surface.

\begin{lemma}\label{lemma:3-squared-and-p}
Let $p>3$ be a prime number,
and let $S$ be a Severi--Brauer surface
over a field $\KK$ of characteristic zero
such that the group $\Aut(S)$
contains a subgroup
$G\cong \mumu_3^2\times\mumu_p$.
Then $S\cong\PP^2$.
\end{lemma}

\begin{proof}
Suppose that the Severi--Brauer surface $S$ is non-trivial.
The group $\mumu_3^2\subset G$ contains four subgroups isomorphic to $\mumu_3$.
According to Lemma~\ref{lemma:unique-point},
each of these subgroups has exactly three fixed
points on~$S_{\bar{\KK}}$. Furthermore, none of the points fixed by two
different subgroups $\mumu_3$ can coincide with each
other, since otherwise $S\cong\PP^2$
by Corollary~\ref{corollary:unique-fixed-point-3-squared}.
On the other hand, since the group $\mumu_p\subset G$ commutes with
$\mumu_3^2$, it maps the fixed points of every element $g$ of $\mumu_3^2$
again to points fixed by~$g$. Therefore, $\mumu_p$ has four distinct
invariant sets of three points on $S_{\bar{\KK}}\cong\PP^2_{\bar{\KK}}$.
Since $p>3$, each of the $12$ points in the union
of these sets is $\mumu_p$-invariant. This is impossible
by Lemma~\ref{lemma:unique-point}.
\end{proof}

The next lemma shows that non-trivial non-balanced semidirect products
do not appear as subgroups of automorphism groups of Severi--Brauer
surfaces over fields of characteristic zero.

\begin{lemma}[{cf. Remark~\ref{remark:not-attained-non-balanced}}]
\label{lemma:bad-semi-direct-product}
Let $n$ be a positive integer divisible only by primes congruent to $1$
modulo $3$, and let $G$ be a
semidirect product of $\mumu_n$ and $\mumu_3$.
Let~$\bar{\KK}$ be an algebraically closed field of characteristic zero.
Suppose that $G\subset\PGL_3\left(\bar{\KK}\right)$.
Then~$G$ is either isomorphic to $\mumu_n\times\mumu_3$,
or is a balanced semidirect product.
\end{lemma}

\begin{proof}
Assume that $G$ is a non-balanced semidirect product.
By Lemma~\ref{lemma:non-balanced-structure}, for some positive integers
$n_1$ and $n_2$ such that $n_1n_2=n$ and~\mbox{$n_2>1$}
there is an isomorphism
$$
G\cong G_1\times\mumu_{n_2},
$$
where
$G_1\cong \mumu_{n_1}\rtimes\mumu_3$.
Let us also assume that~$G$ is not isomorphic to $\mumu_n\times\mumu_3$.
Then $G$, and thus also $G_1$, is not abelian; in particular, we have $n_1>1$.

Consider the projection
$$
\pi\colon \SL_3\left(\bar{\KK}\right)\to\PGL_3\left(\bar{\KK}\right),
$$
and set $\tilde{G}_1=\pi^{-1}(G_1)$.
The group $\tilde{G}_1$ is not abelian, and
its order $3n_1$ is not divisible by~$2$.
Thus the natural three-dimensional representation of $\tilde{G}_1$
is irreducible.

Let $g$ be a generator of the subgroup $\mumu_{n_2}\subset G$,
and choose a preimage
$\tilde{g}$ of $g$ with respect to $\pi$.
Since the order $n_2$ of $g$
is coprime to $3$, we can multiply $\tilde{g}$ by
an appropriate scalar matrix and assume that
the order of $\tilde{g}$ equals $n_2$ as well.
We claim that $\tilde{g}$ commutes with $\tilde{G}_1$.
Indeed, let $\gamma$ be an element of $\tilde{G}_1$.
Since~\mbox{$\pi(\gamma)\in G_1$} commutes with $\pi(\tilde{g})=g$,
the commutator $\kappa=\tilde{g}\gamma\tilde{g}^{-1}\gamma^{-1}$
is a scalar matrix in~\mbox{$\SL_3\left(\bar{\KK}\right)$}.
On the other hand, $\tilde{g}^{n_2}$ is the identity
matrix, so it commutes with $\gamma$. Writing
$$
\tilde{g}^{n_2}\gamma=\kappa^{n_2}\gamma\tilde{g}^{n_2},
$$
one concludes that $\kappa^{n_2}$ is the identity matrix, and thus
$\kappa$ is the identity matrix itself.

We see that $\tilde{g}$ commutes with $\tilde{G}_1$.
Therefore, by Schur's lemma $\tilde{g}$ acts in the three-dimensional
representation by a scalar matrix. This is impossible, because
the image $g$ of~$\tilde{g}$ in $\PGL_3\left(\bar{\KK}\right)$ is non-trivial.
The obtained contradiction shows that $G\cong\mumu_n\times\mumu_3$.
\end{proof}

\begin{remark}
For an alternative proof of Lemma~\ref{lemma:bad-semi-direct-product},
consider two elements~\mbox{$g_1, g_2\in \PGL_3\left(\bar{\KK}\right)$}
such that the orders of $g_i$ are finite, at least one of the orders is
greater than $3$, and the fixed point locus $\Fix(g_i)$
of each $g_i$ consists of
three distinct points. One can notice that $g_1$ and $g_2$
commute with each other
if and only if~\mbox{$\Fix(g_1)=\Fix(g_2)$}.
Applying this criterion to the generators $g_2$
and $g$ of the subgroups~$\mumu_{n_2}$ and~$\mumu_3$ of $G$ appearing
in~\eqref{eq:non-balanced-structure} (and assuming that $\Fix(g_2)$
and $\Fix(g)$
consist of three points), and then to $g$ and a generator
of the subgroup $\mumu_n$ of~$G$, we conclude that $G$ is isomorphic to
a direct product of $\mumu_n$ and $\mumu_3$.
The case when the fixed locus of some of the elements $g_2$ or $g$
consists of a line and an isolated point is
also easy to analyze.
\end{remark}

Let us summarize the above results.

\begin{corollary}\label{corollary:summarize-subgroups}
There exists a field $\KK$ of characteristic zero and
a non-trivial Severi--Brauer surface over $\KK$ such that the group
$\Aut(S)$  contains a finite subgroup isomorphic to $G$ if and only if
there is a positive integer $n$ divisible only by primes congruent to $1$
modulo $3$ such that $G$
is isomorphic either to $\mumu_n$, or to $\mumu_{3n}$, or to a balanced
semidirect product $\mumu_n\rtimes\mumu_3$, or to
the direct product $\mumu_3\times(\mumu_n\rtimes\mumu_3)$,
where the semidirect product is balanced.
\end{corollary}

\begin{proof}
The ``if'' part of the assertion is given by
Corollary~\ref{corollary:elements-attained}
and Lemmas~\ref{lemma:attained-n} and \ref{lemma:attained-3n}.
Let us prove the ``only if'' part.
Let $S$ be a non-trivial Severi--Brauer surface over a field
of characteristic zero, and let $G$ be a finite subgroup of $\Aut(S)$.

To start with, assume that $G$ is abelian.
Then $G$ is isomorphic to the product of its $p$-Sylow subgroups $G(p)$
for all $p$ dividing $|G|$.
By Corollary~\ref{corollary:p-squared-on-P2},
for $p\neq 3$ the group~$G(p)$ is cyclic.
By Corollary~\ref{corollary:3-squared-on-P2},
the group $G(3)$ is isomorphic to
a subgroup of $\mumu_3^2$. Moreover, if~\mbox{$G(3)\cong\mumu_3^2$},
then all the
groups $G(p)$ for $p\neq 3$ are trivial by
Lemma~\ref{lemma:3-squared-and-p}. We conclude that~$G$ is isomorphic
either to $\mumu_3^2$,
or to $\mumu_n$, or to $\mumu_{3n}$, where $n\ge 1$
is the product of~\mbox{$|G(p)|$}
for all prime
divisors of $|G|$ different from $3$.
In the former case,
we can consider~\mbox{$G\cong\mumu_3^2$} as a product of $\mumu_3$ and
a balanced semidirect product $\mumu_1\rtimes\mumu_3\cong\mumu_3$.
In the latter two cases,
all prime divisors $p$ of $n$ are congruent to $1$ modulo $3$ by
Lemma~\ref{lemma:3k2}. This proves the required assertion in the case
when $G$ is abelian.

Now assume that $G$ is not abelian. By Theorem~\ref{theorem:J-3},
the group $G$ contains a normal abelian subgroup $H$ of index $3$.
Since $G$ does not contain elements of order $9$ by Lemma~\ref{lemma:order-9},
one has~\mbox{$G\cong H\rtimes\mumu_3$} by Lemma~\ref{lemma:is-semidirect}.
According to the above argument,
$H$ is isomorphic
either to $\mumu_3^2$,
or to $\mumu_n$, or to $\mumu_{3n}$, where $n$ is divisible only by
primes congruent to $1$ modulo~$3$.
In the first case~\mbox{$|G|=3|H|=27$}, which
is impossible by
Corollary~\ref{corollary:3-squared-on-P2}.
In the second case we have $G\cong\mumu_n\rtimes\mumu_3$,
and $G$ is not isomorphic to $\mumu_n\times\mumu_3$ because it is non-abelian.
Thus $G$ is a balanced semidirect product by
Lemma~\ref{lemma:bad-semi-direct-product}.
In the third case
\begin{equation}\label{eq:G-cong-3n3}
G\cong\mumu_{3n}\rtimes\mumu_3\cong(\mumu_3\times\mumu_n)\rtimes\mumu_3
\cong\mumu_3\times (\mumu_n\rtimes\mumu_3),
\end{equation}
because there are no non-trivial homomorphisms from $\mumu_3$ to
$\mumu_3^*\cong\mumu_2$. Furthermore, the
semidirect product $\mumu_n\rtimes\mumu_3$
on the right hand side of~\eqref{eq:G-cong-3n3}
is not a direct product, because the group $G$ is non-abelian.
Hence this semidirect product is balanced by
Lemma~\ref{lemma:bad-semi-direct-product}. This completes the proof
in the case when $G$ is non-abelian.
\end{proof}

The following assertion is implied by the results of~\cite{Shramov-SB}
and~\cite{Shramov-Cubics}.

\begin{proposition}\label{proposition:cubic}
Let $S$ be a non-trivial Severi--Brauer surface over a field of characteristic
zero, and let $G$ be a finite subgroup of $\Bir(S)$. Then either
$G$ is isomorphic to a subgroup of $\Aut(S)$, or $G$ is a subgroup
of~$\mumu_3^3$. Furthermore, there exists a non-trivial Severi--Brauer
surface over a field of characteristic zero whose birational
automorphism group contains a subgroup isomorphic to~$\mumu_3^3$.
\end{proposition}

\begin{proof}
By \cite[Proposition~3.7]{Shramov-SB},
the group $G$ is either a subgroup of $\mumu_3^3$, or is isomorphic to 
a subgroup of $\Aut(S)$, or is isomorphic to
a subgroup of $\Aut(S^{op})$, where~$S^{op}$ is the Severi--Brauer surface corresponding to the  
central simple algebra opposite to the one corresponding to~$S$.
In the latter case $G$ is also  isomorphic to
a subgroup of~\mbox{$\Aut(S)$} by \cite[Corollary~4.5]{Shramov-SB}.
A Severi--Brauer surface $S$ with $\mumu_3^3\subset\Bir(S)$ exists by~\mbox{\cite[Theorem~1.2]{Shramov-Cubics}}.
\end{proof}

Now we are ready to prove our main result.

\begin{proof}[Proof of Theorem~\ref{theorem:main}]
Assertion~(i) is given by Corollary~\ref{corollary:order}
and Proposition~\ref{proposition:cubic}.
Assertion~(ii) is given by Corollary~\ref{corollary:summarize-subgroups}.
Assertion~(iii) follows from Proposition~\ref{proposition:cubic}.
\end{proof}

\begin{remark}
There is a number of results on finite subgroups of multiplicative groups of division algebras, starting from the classical
papers \cite{Herstein} and \cite{Amitsur}.
They can hardly be used to conclude
anything about automorphism groups of Severi--Brauer varieties, since
a finite group in the quotient $A^*/\KK^*$ is not necessarily an image of a finite subgroup of $A^*$. However, every finite
subgroup in $A^*$ projects to a finite subgroup of $A^*/\KK^*$.
It would be interesting to find out which results on finite subgroups of multiplicative groups of division algebras
can be recovered from geometric properties of Severi--Brauer varieties.
\end{remark}

\section{Field of rational numbers}
\label{section:Q}

In this section  we discuss finite
groups acting by birational automorphisms of non-trivial Severi--Brauer
surfaces over~$\QQ$ and prove Corollary~\ref{corollary:Q}.

It is well known that the group $\PGL_3(\QQ)$ does not contain
elements of prime order~\mbox{$p\ge 5$},
see e.g.~\mbox{\cite[\S1]{DolgachevIskovskikh-prime}}.
The following assertion is a generalization of this result for
automorphism groups of Severi--Brauer surfaces.

\begin{lemma}
\label{lemma:7-over-Q}
Let $S$ be a Severi--Brauer surface over the field $\QQ$.
Then the group~\mbox{$\Aut(S)$}
does not contains elements of prime order $p\ge 5$.
\end{lemma}

We give two slightly different proofs of Lemma~\ref{lemma:7-over-Q}.
The first of them uses the same approach as the proofs of
Lemmas~\ref{lemma:3k2} and~\ref{lemma:order-9}.

\begin{proof}[First proof of Lemma~\ref{lemma:7-over-Q}]
Suppose that there is an element $g\in\Aut(S)$ of order $p\ge 5$.
We may assume that the Severi--Brauer surface $S$ is non-trivial. 
Consider the action of $g$ on $S_{\bar{\QQ}}\cong\PP^2_{\bar{\QQ}}$.
Let $\tilde{g}$ be a preimage of $g$ under the natural projection
$$
\pi\colon\SL_3\left(\bar{\QQ}\right)\to\PGL_3\left(\bar{\QQ}\right).
$$
Since $p$ is not divisible by $3$, we can multiply $\tilde{g}$ by
an appropriate scalar matrix so that the order of $\tilde{g}$ equals $p$.
Moreover, the element~$\tilde{g}$ of order $p$ such
that $\pi(\tilde{g})=g$ is unique
in~$\SL_3\left(\bar{\QQ}\right)$, and
hence it is defined over the field~$\QQ$.

Let $\xi_1$, $\xi_2$, and $\xi_3$ be the eigen-values of $\tilde{g}$;
these are $p$-th roots of unity.
According to Lemma~\ref{lemma:unique-point}, the element $g$
has exactly three fixed points on $\PP^2_{\bar{\QQ}}$, and these points
are transitively permuted by the group $\Gal\left(\bar{\QQ}/\QQ\right)$.
Thus the roots of unity $\xi_1$, $\xi_2$, and~$\xi_3$
are pairwise different and form a $\Gal\left(\bar{\QQ}/\QQ\right)$-orbit.
However, the group $\Gal\left(\bar{\QQ}/\QQ\right)$
acts transitively on the $p-1$ primitive roots of unity of degree $p$ by
Theorem~\ref{theorem:cyclotomic}. Since~\mbox{$p-1>3$},
this gives a contradiction.
\end{proof}

The second proof uses the properties of central simple algebras.

\begin{proof}[Second proof of Lemma~\ref{lemma:7-over-Q}]
Suppose that there is an element $g\in\Aut(S)$ of order $p\ge 5$.
We may assume that the Severi--Brauer surface $S$ is non-trivial. 
Let $A$ be the central simple algebra of dimension $9$ over $\QQ$
corresponding to the surface $S$.
By Lemma~\ref{lemma:SB-Aut} there exists an element $x\in A^*$
such that~\mbox{$x\not\in\QQ$} and $x^p=b$ for some $b\in\QQ$.

Since $S$ is a non-trivial Severi--Brauer surface, $A$ is a division algebra
by Lemma~\ref{lemma:matrix-or-division}. 
Thus $x$ generates a
subfield $\LL\supsetneq\QQ$ inside $A$.
The field $\LL$ is contained in some maximal subfield
of $A$. On the other hand, any maximal subfield of $A$
has degree $3$ over~$\QQ$, see~\mbox{\cite[\S\,VIII.10.3]{Bourbaki}}.
Hence $\LL$ also has
degree $3$ over~$\QQ$ (and is maximal itself).
Therefore, there exists an irreducible polynomial $F(t)$
of degree~$3$ over~$\QQ$ such that~$x$ is a root of~$F(t)$.
Thus $F(t)$ divides the polynomial~\mbox{$t^p-b$}.
Note that if $b$ is not a $p$-th power of a rational number,
then the polynomial $t^p-b$ is irreducible over~$\QQ$, see
for instance~\mbox{\cite[Theorem~VIII.9.16]{Lang}}.
Since $p>3$ by assumption, one cannot have~\mbox{$F(t)=t^p-b$}, and hence
the polynomial $t^p-b$ must be reducible.
Thus we see that
\begin{equation}\label{eq:bcp}
b=c^p
\end{equation}
for some $c\in\QQ$,
and the roots of $F(t)$
have the form~\mbox{$x=\xi_1c$},~$\xi_2c$, and~$\xi_3c$,
where~$\xi_1$,~$\xi_2$, and~$\xi_3$ are pairwise different
$p$-th roots of unity.
We see that $\xi_1$, $\xi_2$, and $\xi_3$ form
a $\Gal\left(\bar{\QQ}/\QQ\right)$-orbit.
This is impossible by Theorem~\ref{theorem:cyclotomic}.
Alternatively, one can conclude from~\eqref{eq:bcp}
that the polynomial~\mbox{$F(ct)$} divides the $p$-th cyclotomic polynomial
$\Phi_p(t)$, and use the fact that the latter is irreducible over~$\QQ$. 
\end{proof}

For an alternative proof (and a more general statement)
of Lemma~\ref{lemma:7-over-Q} we refer the reader to
\cite[Theorem~6]{Serre}.

Finally, we prove Corollary~\ref{corollary:Q}.

\begin{proof}[Proof of Corollary~\ref{corollary:Q}]
By Proposition~\ref{proposition:cubic}, we may assume that the group
$G$ is contained in $\Aut(S)$. Then the order of $G$
is divisible only by primes not exceeding $3$
by Lemma~\ref{lemma:7-over-Q}.
On the other hand, we know
from Theorem~\ref{theorem:main}(i) (or from Lemma~\ref{lemma:3k2})
that the order of $G$ is odd.
Therefore, in this case $G$ is a $3$-group, and thus it is a
subgroup of~$\mumu_3^2$ by Corollary~\ref{corollary:3-squared-on-P2}.
\end{proof}

\begin{remark}\label{remark:referee}
Using the same arguments as in the proofs of Lemma~\ref{lemma:7-over-Q} and Corollary~\ref{corollary:Q},
one can generalize both of these assertions to arbitrary fields that contain no roots of
unity of any prime degree~\mbox{$p\ge 5$}.
\end{remark}

\end{document}